\documentclass[12pt]{article}

\usepackage[T1]{fontenc}
\usepackage[english]{babel}
\usepackage[a4paper,margin=0.8in]{geometry}
\usepackage{amsmath,amssymb,amsthm,mathtools,bm}
\usepackage{booktabs,array,tabularx}
\usepackage{microtype}
\usepackage{hyperref}
\usepackage{enumitem}
\usepackage{cleveref}
\usepackage{caption}
\usepackage{float}
\usepackage{tikz}
\usepackage{graphicx}
\usetikzlibrary{positioning,fit,calc,decorations.pathreplacing,backgrounds}
\allowdisplaybreaks[1]
\hypersetup{
  hidelinks,
  pdftitle={Covering powers of cycles with equivalence subgraphs},
  pdfauthor={Zhanping Yang},
  pdfsubject={Equivalence coverings of powers of cycles},
  pdfkeywords={equivalence covering, powers of cycles, clique partitions, binary encoding, matrix rank, squares and cubes}
}

\newtheorem{theorem}{Theorem}[section]
\newtheorem{lemma}[theorem]{Lemma}
\newtheorem{conjecture}[theorem]{Conjecture}

\newtheorem{proposition}[theorem]{Proposition}

\newtheorem{claim}[theorem]{Claim}

\theoremstyle{remark}
\newtheorem{remark}[theorem]{Remark}

\newcommand{\eqc}{\operatorname{eq}}
\newcommand{\rank}{\operatorname{rank}}

\newcommand{\ceil}[1]{\left\lceil #1\right\rceil}

\newcommand{\R}{\mathbb R}

\newcommand{\cyclecover}[1]{%
	\begin{tikzpicture}[
		dot/.style={circle,draw=black,fill=white,
			inner sep=0pt,minimum size=4.6pt,line width=0.55pt},
		used/.style={dot,fill=black},
		every node/.style={font=\small}
		]
		\def\R{1.52}
		\foreach \i in {1,...,12}{
			\pgfmathsetmacro{\ang}{135-30*\i}
			\coordinate (v\i) at (\ang:\R);
		}
		\draw[gray!45,densely dashed,line width=0.5pt]
		(0,0) circle (\R);
		
		\draw[gray!80,line width=0.85pt,
		rotate around={-30:(0.643,1.113)}]
		(0.643,1.113) ellipse[x radius=1.42cm,y radius=0.48cm];
		\draw[gray!80,line width=0.85pt,
		rotate around={30:(0.643,-1.113)}]
		(0.643,-1.113) ellipse[x radius=1.42cm,y radius=0.48cm];
		\draw[gray!80,line width=0.85pt,
		rotate around={90:(-1.286,0)}]
		(-1.286,0) ellipse[x radius=1.42cm,y radius=0.48cm];
		
		\ifcase#1\relax
		\foreach \a/\b in {
			1/2,1/3,1/4,2/3,2/4,3/4,
			5/6,5/7,5/8,6/7,6/8,7/8,
			9/10,9/11,9/12,10/11,10/12,11/12}{
			\draw[line width=0.55pt] (v\a)--(v\b);
		}
		\or
		\foreach \a/\b in {
			3/4,5/6,3/5,3/6,4/5,4/6,
			7/8,9/10,7/9,7/10,8/9,8/10,
			11/12,1/2,11/1,11/2,12/1,12/2}{
			\draw[line width=0.55pt] (v\a)--(v\b);
		}
		\or
		\foreach \a/\b in {2/5,4/7,6/9,8/11,10/1,12/3}{
			\draw[line width=0.55pt] (v\a)--(v\b);
		}
		\fi
		
		\foreach \i in {1,...,12}{\node[used] at (v\i) {};}
		
		\node at (60:2.12) {$B_0$};
		\node at (-60:2.12) {$B_1$};
		\node at (180:2.10) {$B_2$};
		
		\ifcase#1\relax
		\node[font=\scriptsize] at (106:1.10) {$v_1$};
		\node[font=\scriptsize] at (12:1.12) {$v_4$};
		\node[font=\scriptsize] at (-12:1.12) {$v_5$};
		\node[font=\scriptsize] at (-106:1.10) {$v_8$};
		\node[font=\scriptsize] at (-134:1.08) {$v_9$};
		\node[font=\scriptsize] at (134:1.08) {$v_{12}$};
		\node at (0,-2.50) {(a) Group edges};
		\or
		\node at (0,-2.50) {(b) $\ell=1$};
		\or
		\node at (0,-2.50) {(c) $\ell=0$};
		\fi
	\end{tikzpicture}%
}

\crefname{theorem}{theorem}{theorems}
\crefname{lemma}{lemma}{lemmas}
\crefname{conjecture}{conjecture}{conjectures}
\crefname{corollary}{corollary}{corollaries}
\crefname{proposition}{proposition}{propositions}
\crefname{remark}{remark}{remarks}

\title{Covering powers of cycles with equivalence subgraphs}
\author{
  Zhanping Yang\thanks{Corresponding author. Center for Discrete Mathematics, Fuzhou University, Fuzhou, Fujian 350108, China. Research supported by National Key R\&D Program of China (Grant No. 2023YFA1010202) and National Natural Science Foundation of China (Grant No. 12371342). Email: \texttt{yangzp@fzu.edu.cn}.}
}
\date{}

\begin{document}
\maketitle

\begin{abstract}
Let $C_n^k$ denote the $k$th power of the cycle $C_n$. The equivalence covering number $\operatorname{eq}(G)$,
studied by Alon, is the minimum number of equivalence subgraphs needed to cover all edges of $G$. King conjectured that $\operatorname{eq}(C_n^k)=\Omega(k)$ for sufficiently large $n$ relative to $k$. We study $\operatorname{eq}(C_n^k)$ and prove that, for every $n \geq 2k + 2$, we have $1 + \lceil\log_2(k+1)\rceil \leq \operatorname{eq}(C_n^k) \leq 1 + 2\lceil\log_2(k+1)\rceil$. We further show that, whenever $n\ge(k+1)\max\{41,5\lceil\log_2(k+1)\rceil\}$,
\begin{equation*}
\operatorname{eq}(C_n^k)\in\left\{1+\lceil\log_2(k+1)\rceil,\,2+\lceil\log_2(k+1)\rceil\right\}.
\end{equation*}
Moreover, for every integer $n\ge \max\{2k+2,k^2(k+2)\}$, the following exact values hold in the indicated cases:
\begin{equation*}
	\operatorname{eq}(C_n^k)=\begin{cases}
		1+\lceil\log_2(k+1)\rceil,
		& k+1\neq2^{\lceil\log_2(k+1)\rceil}
		\text{ or }(k+1)\mid n,\\[3pt]
		2+\lceil\log_2(k+1)\rceil,
		& k+1=2^{\lceil\log_2(k+1)\rceil}
		\text{ and }2\nmid n.
	\end{cases}
\end{equation*}
\end{abstract}



\section{Introduction}\label{sec:intro}
For a graph $G$, an \emph{equivalence subgraph} is a disjoint union of cliques, with isolated vertices allowed. The \emph{equivalence covering number} $\eqc(G)$ is the minimum number of equivalence subgraphs of $G$ whose edge sets cover $E(G)$. Alon~\cite{Alo1986} established bounds for the equivalence covering number of complements of graphs with bounded degree. More precisely, if $H$ is an $n$-vertex graph with minimum degree at least $1$ and maximum degree at most $d$, then
\begin{equation*}
\log_2\frac{n}{d}\le\eqc(\overline{H})\le2e^2(d+1)^2\ln n.
\end{equation*}
In particular, for every fixed $d$, we have $\eqc(\overline{H})=\Theta(\log n)$.

This parameter is closely related to the classical problem of covering edges by cliques. Erd\H{o}s, Goodman, and P\'osa~\cite{EGP1966} proved that every $n$-vertex graph admits
an edge partition into at most $\lfloor n^2/4\rfloor$ cliques, using only edges and triangles. This bound is best possible, as shown by the balanced complete bipartite graph. Lov\'asz~\cite{Lov1968} refined the clique-covering bound in terms of the number of edges. More precisely, if $G$ has $m$ edges and $t$ is the largest integer satisfying $t(t-1)\le \binom{n}{2}-m$, then $G$ can be covered by at most $\binom{n}{2}-m+t$ cliques. Orlin~\cite{Orl1977} studied coverings by complete and complete bipartite graphs, while Roberts~\cite{Rob1985} investigated applications of edge coverings by cliques.

Several further results concern clique coverings and clique partitions. Pullman~\cite{Pul1984} studied algorithmic questions for clique coverings, while Pullman, Shank, and Wallis~\cite{PSW1982} investigated partitions into maximal cliques. Erd\H{o}s, Faudree, and Ordman~\cite{EFO1988} showed that the clique partition number of an $n$-vertex graph can be larger than its clique covering number by a factor of order $n^2$. De Caen, Erd\H{o}s, Pullman, and Wormald~\cite{DEP1986} proved that the maximum product of the clique covering numbers of an $n$-vertex graph and its complement is asymptotic to $n^4/256$. In another direction, Erd\H{o}s and Pyber~\cite{EP1997} proved that the edge set of every $n$-vertex graph can be partitioned into complete bipartite graphs so that each vertex belongs to at most $O(n/\log n)$ of them. These results highlight the differences between covering and partitioning edges by various classes of subgraphs.

For particular graph classes, de Caen, Gregory, and Pullman~\cite{DGP1985} obtained logarithmic bounds for clique covers of complements of paths and cycles. Klotz~\cite{Klo1989} investigated connections between clique covers and graph coloring, and Prisner~\cite{Pri1995} studied clique covers and partitions in generalizations of line graphs.

For the equivalence covering number, Esperet, Gimbel, and King~\cite{EGK2010} established the following bounds for the line graph $L(G)$ of a graph $G$:
\begin{equation*}
\frac{1}{3}\log_2\log_2\chi(G)< \eqc(L(G))\le 2\log_2\log_2\chi(G)+2,
\end{equation*}
where $\chi(G)$ denotes the chromatic number of $G$. In particular, they showed that the equivalence covering number of line graphs is unbounded even when the underlying graphs are triangle-free. They also proved that deciding whether $\eqc(L(G))\le3$ is NP-complete for triangle-free graphs $G$.

Golumbic, Morgenstern, and Rajendraprasad~\cite{GMR2018} obtained bounds on equivalence coverings for certain edge-intersection graphs of paths in a grid. A related formulation uses \emph{tessellations}, which are partitions of the vertex set into cliques; Abreu et al.~\cite{ACF2022} studied an extension called total tessellation covering, with additional vertex-coloring conditions.

We study the equivalence covering number of powers of cycles. The $k$th power $C_n^k$ has vertex set $V(C_n)$, with two distinct vertices adjacent whenever their distance in $C_n$ is at most $k$. King~\cite{Kin2011} proposed linear growth of the equivalence covering number of $C_n^k$ in $k$, even when $n$ is large compared with $k$. The question was recorded in the following form.
\begin{conjecture}[\textbf{King}~\cite{Kin2011}]\label{conj:king}
Given $k$ and $n$, the graph $C_n^k$ has equivalence covering number $\Omega(k)$.
\end{conjecture}
Our first result shows that, for \(n \ge 2k+2\), the growth of \(\eqc(C_n^k)\) in \(k\) is logarithmic rather than linear.
\begin{theorem}\label{thm:main}
	Let $n\ge3$ and $k\ge1$ be integers. If $n\geq 2k+2$, then
	\begin{equation*}
	1+\ceil{\log_2(k+1)}\leq\eqc(C_n^k)\le1+2\ceil{\log_2(k+1)}.
	\end{equation*}
	In particular, if $n\leq 2k+1$, then $\eqc(C_n^k)=1$.
\end{theorem}
We next determine the exact value of $\eqc(C_n^k)$ for certain values of $k$ and $n$.
\begin{theorem}\label{thm:exact-values}
Let $k\ge1$ be an integer and $c_k=\max\{2k+2,k^2(k+2)\}$. For every integer $n\ge c_k$, the following exact values hold in the indicated cases:
	\begin{equation*}
	\eqc(C_n^k)=\begin{cases}
	1+\lceil\log_2(k+1)\rceil,
	& k+1\neq2^{\lceil\log_2(k+1)\rceil}
	\text{ or }(k+1)\mid n,\\[3pt]
	2+\lceil\log_2(k+1)\rceil,
	& k+1=2^{\lceil\log_2(k+1)\rceil}
	\text{ and }2\nmid n.
\end{cases}
	\end{equation*}
\end{theorem}
We further show that the bounds differ by at most $1$ whenever $n=\Omega(k\log_2 k)$.
\begin{theorem}\label{thm:main-gap}
	For integers $k\ge1$ and $n\ge(k+1)\max\{41,5\lceil\log_2(k+1)\rceil\}$, we have
	\begin{equation*}
		\operatorname{eq}(C_n^k)\in\left\{1+\ceil{\log_2(k+1)}, 2 +\ceil{\log_2(k+1)}\right\}.
	\end{equation*}
\end{theorem}
The paper is organized as follows. Section \ref{sec:1.2} establishes the general logarithmic bounds for powers of cycles. Section \ref{sec:1.3} proves Theorem \ref{thm:exact-values} using induced subgraph embeddings and a lower bound for odd $n$. Section \ref{sec:1.4} improves the general upper bound and proves Theorem \ref{thm:main-gap}.
\section{Proof of Theorem~\ref{thm:main}} \label{sec:1.2}
In this section, we establish the general logarithmic bounds for $\eqc(C_n^k)$. We first prove a lower bound, then give a construction that gives an exact value when $(k+1)\mid n$ and an upper bound for general $n$. When preparing this paper, we notice that Ferudun~\cite{Fer2026} also obtained results similar to Lemmas~\ref{lem:lower} and~\ref{lem:comparison}. We present our proofs here.

Write the vertices of $C_n$ as $v_0,\ldots,v_{n-1}$, with subscripts
taken modulo $n$. When $n\le2k+1$, every pair of vertices has cyclic
distance at most $k$, so $C_n^k=K_n$ and $\eqc(C_n^k)=1$.
We therefore assume $n\ge2k+2$ throughout the remainder of this section.
\begin{lemma}\label{lem:monotonicity}
	If $G$ is an induced subgraph of $H$, then $\eqc(G)\le\eqc(H)$.
\end{lemma}
\begin{proof}[\bf Proof]
	Let $t=\operatorname{eq}(H)$, and let $H_1,\ldots,H_t$ be equivalence subgraphs of $H$ such that
	\begin{equation*}
		E(H)=\bigcup_{i=1}^t E(H_i).
	\end{equation*}
	Since $G=H[V(G)]$, each $H_i[V(G)]$ is an equivalence subgraph of $G$, and
	\begin{equation*}
		E(G)=\bigcup_{i=1}^t E(H_i[V(G)]).
	\end{equation*}
	Hence $\operatorname{eq}(G)\le t$.
\end{proof}
We next establish a lower bound for the equivalence covering number of powers of cycles.
\begin{lemma}\label{lem:lower}
	If $n\ge2k+2$, then
	$\eqc(C_n^k)\ge1+\ceil{\log_2(k+1)}$.
\end{lemma}
\begin{proof}[\bf Proof]
	Let $t=\operatorname{eq}(C_n^k)$, and let $G_1,\ldots,G_t$ be spanning equivalence subgraphs of $C_n^k$ such that
	\begin{equation*}
		E(C_n^k)=\bigcup_{s=1}^t E(G_s).
	\end{equation*}
	For each $s\in[t]$, assign distinct integers to the components of $G_s$, and let $c_s(v)$
	denote the integer assigned to the component containing $v$. For any distinct vertices $u,v$,
	\begin{equation*}
		\begin{cases}
			c_s(u)-c_s(v)=0\text{ for some }s\in[t],&uv\in E(C_n^k),\\
			c_s(u)-c_s(v)\ne0\text{ for every }s\in[t],&uv\notin E(C_n^k).
		\end{cases}
	\end{equation*}
	For $u,v\in V(C_n^k)$, define
	\begin{equation*}
		\varphi(u,v)=\prod_{s=1}^t(c_s(u)-c_s(v)).
	\end{equation*}
	Then if $u=v$ or $uv\in E(C_n^k)$, then $\varphi(u,v)=0$, and if $u,v$ are distinct and nonadjacent, then $\varphi(u,v)\ne0$. Consider the $(2k+2)\times(2k+2)$ matrix $M=\bigl(\varphi(v_i,v_{j+k+1})\bigr)_{0\le i,j\le 2k+1}$, where subscripts are taken modulo $n$.
	\begin{claim}\label{claim:diagonal}
		$M_{ii}\ne 0$ for every $0\le i\le 2k+1$.
	\end{claim}
	\begin{proof}
		Since $n\ge 2k+2$, the cyclic distance between $v_i$ and $v_{i+k+1}$ is $k+1$. Hence
		$v_iv_{i+k+1}\notin E(C_n^k)$, and thus $M_{ii}=\varphi(v_i,v_{i+k+1})\ne 0$.
	\end{proof}
	\begin{claim}\label{claim:triangular}
		$M_{ij}=0$ for every $0\le j<i\le 2k+1$.
	\end{claim}
	\begin{proof}
		Since $1\le i-j\le 2k+1$, we have $-k\le k+1+j-i\le k$. Hence $v_i=v_{j+k+1}$ or $v_iv_{j+k+1}\in E(C_n^k)$, and thus $M_{ij}=\varphi(v_i,v_{j+k+1})=0$.
	\end{proof}
	By Claims~\ref{claim:diagonal} and~\ref{claim:triangular}, $M$ is upper triangular with nonzero diagonal. Hence $\rank M=2k+2$. On the other hand, expanding the product gives
	\begin{equation*}
		M_{ij}=\sum_{S\subseteq[t]}(-1)^{t-|S|}\left(\prod_{s\in S}c_s(v_i)\right)\left(\prod_{s\in[t]\setminus S}c_s(v_{j+k+1})\right).
	\end{equation*}
	Let $U\in\mathbb R^{(2k+2)\times 2^t}$ and $V\in\mathbb R^{2^t\times(2k+2)}$ be matrices whose columns and rows, respectively, are indexed by the subsets $S\subseteq[t]$, with
	\begin{equation*}
		U_{i,S}=\prod_{s\in S}c_s(v_i),
		\qquad
		V_{S,j}=(-1)^{t-|S|}
		\prod_{s\in[t]\setminus S}c_s(v_{j+k+1}).
	\end{equation*}
	Then $M=UV$, and consequently, 
	\begin{equation*}
		2k+2=\operatorname{rank}M\le\min\{\operatorname{rank}U,\operatorname{rank}V\}\le 2^t.
	\end{equation*}
	Therefore, we have $t\ge 1+\lceil\log_2(k+1)\rceil$, as required.
\end{proof}
We next give a construction that will be used to establish the upper bounds.
\begin{lemma}\label{lem:comparison}
	Let $A=\{A_0,\ldots,A_k\}$ and $B=\{B_0,\ldots,B_k\}$ be vertex-disjoint cliques in a graph $G$. Suppose that $A_aB_b\in E(G)$ whenever $a>b$. Then the edges $A_aB_b$ with $a>b$ can be covered by $\lceil\log_2(k+1)\rceil$ equivalence subgraphs of $G$.
\end{lemma}
\begin{proof}[\bf Proof]
	For $0\le \ell<\lceil\log_2(k+1)\rceil$ and $p\ge0$, define
	\begin{equation}\label{eq:Ilp1}
		I_{\ell,p}^1=\left\{a\in\{0,\ldots,k\}:\left\lfloor\frac{a}{2^\ell}\right\rfloor=2p+1\right\},
	\end{equation}
	and
	\begin{equation}\label{eq:Ilp0}
		I_{\ell,p}^0=\left\{b\in\{0,\ldots,k\}:\left\lfloor\frac{b}{2^\ell}\right\rfloor=2p\right\}.
	\end{equation}
	Let $Q_{\ell,p}=\{A_a:a\in I_{\ell,p}^1\}\cup\{B_b:b\in I_{\ell,p}^0\}$. For each fixed $\ell$, the nonempty sets $Q_{\ell,p}$ are pairwise vertex-disjoint. Moreover, if $a\in I_{\ell,p}^1$ and
	$b\in I_{\ell,p}^0$, then $a\ge (2p+1)2^\ell>b$. Since $A$ and $B$ are cliques, each $Q_{\ell,p}$ is a clique in $G$. Thus these cliques, together with the remaining vertices as singleton components, form an equivalence subgraph of $G$.
	
	Now let $0\le b<a\le k$, and choose the largest integer $\ell$ such that $\left\lfloor{a}/{2^\ell}\right\rfloor>\left\lfloor{b}/{2^\ell}\right\rfloor$. Clearly, $\ell<\lceil\log_2(k+1)\rceil$. By the maximality of $\ell$, we have
	\begin{equation*}
		\left\lfloor\frac{a}{2^{\ell+1}}\right\rfloor=\left\lfloor\frac{b}{2^{\ell+1}}\right\rfloor=p.
	\end{equation*}
	for some $p\ge0$. Thus $\left\lfloor{a}/{2^\ell}\right\rfloor,\left\lfloor{b}/{2^\ell}\right\rfloor\in\{2p,2p+1\}$. Consequently,
	\begin{equation*}
		\left\lfloor\frac{a}{2^\ell}\right\rfloor=2p+1,
		\qquad
		\left\lfloor\frac{b}{2^\ell}\right\rfloor=2p.
	\end{equation*}
	Hence $A_a,B_b\in Q_{\ell,p}$, so $A_aB_b$ is covered by the corresponding equivalence subgraph. Therefore all required edges can be covered by $\lceil\log_2(k+1)\rceil$ equivalence subgraphs.
\end{proof}
For $k=3$, the construction gives two equivalence subgraphs. The first has $Q_{1,0}=\{A_2,A_3,B_0,B_1\}$ as its only nontrivial component, while the second has $Q_{0,0}=\{A_1,B_0\}$ and $Q_{0,1}=\{A_3,B_2\}$ as its nontrivial components. These two equivalence subgraphs cover all edges $A_aB_b$ with $a>b$, as illustrated in Figure~\ref{fig:k3-cover}.
\begin{figure}[H]
	\centering
	
	\begin{minipage}[t]{0.32\textwidth}
		\centering
		\begin{tikzpicture}[
			x=0.95cm,y=0.78cm,
			dot/.style={circle,draw,fill=white,inner sep=0pt,minimum size=4.8pt},
			used/.style={circle,draw,fill=black,inner sep=0pt,minimum size=4.8pt},
			lab/.style={font=\small}
			]
			
			% vertices on side A
			\node[dot]  (A0) at (0,3) {};
			\node[used] (A1) at (0,2) {};
			\node[used] (A2) at (0,1) {};
			\node[used] (A3) at (0,0) {};
			
			% vertices on side B
			\node[used] (B0) at (2.8,3) {};
			\node[used] (B1) at (2.8,2) {};
			\node[used] (B2) at (2.8,1) {};
			\node[dot]  (B3) at (2.8,0) {};
			
			% labels
			\node[lab,left=3pt of A0] {$A_0$};
			\node[lab,left=3pt of A1] {$A_1$};
			\node[lab,left=3pt of A2] {$A_2$};
			\node[lab,left=3pt of A3] {$A_3$};
			
			\node[lab,right=3pt of B0] {$B_0$};
			\node[lab,right=3pt of B1] {$B_1$};
			\node[lab,right=3pt of B2] {$B_2$};
			\node[lab,right=3pt of B3] {$B_3$};
			
			% target edges
			\draw (A1)--(B0);
			\draw (A2)--(B0);
			\draw (A2)--(B1);
			\draw (A3)--(B0);
			\draw (A3)--(B1);
			\draw (A3)--(B2);
			
			\node[font=\footnotesize] at (1.4,-0.85) {$A_aB_b\ (a>b)$};
			
		\end{tikzpicture}
	\end{minipage}
	\hfill
	\begin{minipage}[t]{0.32\textwidth}
		\centering
		\begin{tikzpicture}[
			x=0.95cm,y=0.78cm,
			dot/.style={circle,draw,fill=white,inner sep=0pt,minimum size=4.8pt},
			used/.style={circle,draw,fill=black,inner sep=0pt,minimum size=4.8pt},
			lab/.style={font=\small}
			]
			
			% vertices on side A
			\node[dot]  (A0) at (0,3) {};
			\node[dot]  (A1) at (0,2) {};
			\node[used] (A2) at (0,1) {};
			\node[used] (A3) at (0,0) {};
			
			% vertices on side B
			\node[used] (B0) at (2.8,3) {};
			\node[used] (B1) at (2.8,2) {};
			\node[dot]  (B2) at (2.8,1) {};
			\node[dot]  (B3) at (2.8,0) {};
			
			% labels
			\node[lab,left=3pt of A0] {$A_0$};
			\node[lab,left=3pt of A1] {$A_1$};
			\node[lab,left=3pt of A2] {$A_2$};
			\node[lab,left=3pt of A3] {$A_3$};
			
			\node[lab,right=3pt of B0] {$B_0$};
			\node[lab,right=3pt of B1] {$B_1$};
			\node[lab,right=3pt of B2] {$B_2$};
			\node[lab,right=3pt of B3] {$B_3$};
			
			% edges in Q_{1,0}
			\draw (A2)--(A3);
			\draw (B0)--(B1);
			\draw (A2)--(B0);
			\draw (A2)--(B1);
			\draw (A3)--(B0);
			\draw (A3)--(B1);
			
			\node[font=\footnotesize]at (1.4,-1.00)
			{$Q_{1,0}=\{A_2,A_3,B_0,B_1\}$};
			
		\end{tikzpicture}
	\end{minipage}
	\hfill
	\begin{minipage}[t]{0.32\textwidth}
		\centering
		\begin{tikzpicture}[
			x=0.95cm,y=0.78cm,
			dot/.style={circle,draw,fill=white,inner sep=0pt,minimum size=4.8pt},
			used/.style={circle,draw,fill=black,inner sep=0pt,minimum size=4.8pt},
			lab/.style={font=\small}
			]
			
			% vertices on side A
			\node[dot]  (A0) at (0,3) {};
			\node[used] (A1) at (0,2) {};
			\node[dot]  (A2) at (0,1) {};
			\node[used] (A3) at (0,0) {};
			
			% vertices on side B
			\node[used] (B0) at (2.8,3) {};
			\node[dot]  (B1) at (2.8,2) {};
			\node[used] (B2) at (2.8,1) {};
			\node[dot]  (B3) at (2.8,0) {};
			
			% labels
			\node[lab,left=3pt of A0] {$A_0$};
			\node[lab,left=3pt of A1] {$A_1$};
			\node[lab,left=3pt of A2] {$A_2$};
			\node[lab,left=3pt of A3] {$A_3$};
			
			\node[lab,right=3pt of B0] {$B_0$};
			\node[lab,right=3pt of B1] {$B_1$};
			\node[lab,right=3pt of B2] {$B_2$};
			\node[lab,right=3pt of B3] {$B_3$};
			
			% edges in Q_{0,0} and Q_{0,1}
			\draw (A1)--(B0);
			\draw (A3)--(B2);
			\node[font=\footnotesize] at (1.18,-0.98)
			{$Q_{0,0}=\{A_1,B_0\}, Q_{0,1}=\{A_3,B_2\}$};
		\end{tikzpicture}
	\end{minipage}
	\caption{The construction in Lemma~\ref{lem:comparison} for $k=3$.}
	\label{fig:k3-cover}
\end{figure}
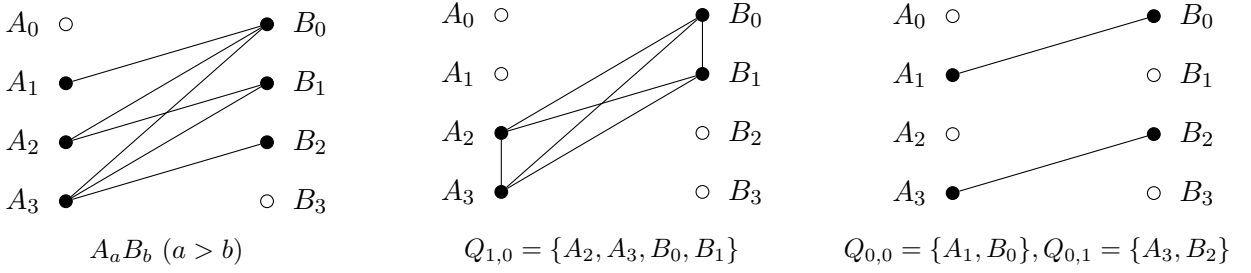

We now use Lemma~\ref{lem:comparison} to determine the exact value of $\eqc(C_n^k)$ when $k+1\mid n$.
\begin{theorem}\label{thm:main-divisible}
	For all integers $k\ge1$ and $q\ge2$, we have
	\begin{equation*}
		\eqc\!\left(C_{q(k+1)}^k\right)=1+\ceil{\log_2(k+1)}.
	\end{equation*}
\end{theorem}
\begin{proof}[\bf Proof]
	Write $B_j=\{v_{j(k+1)+a}:0\le a\le k\}$, where $0\le j<q$ with the index $j$ taken modulo $q$. Thus $V(C_n^k)$ is partitioned into $q$ consecutive groups $B_0,\dots,B_{q-1}$, each of size $k+1$. 
	
	Since any two vertices in the same $B_j$ have cyclic distance at most $k$, each $B_j$ induces a clique in $C_n^k$. Hence all edges inside the groups $B_j$ are covered by one equivalence subgraph.

	We next consider edges joining distinct groups. Let $x=v_{j(k+1)+a}\in B_j$ and $y=v_{j'(k+1)+b}\in B_{j'}$. Suppose that $B_j$ and $B_{j'}$ are nonadjacent in the cyclic order.
	As illustrated in Figure~\ref{fig:cyclic-partition}, each of the two $x$--$y$ paths in $C_n$ contains an entire group of $k+1$ vertices. Hence $d_{C_n}(x,y)>k$, and thus $xy\notin E(C_n^k)$.
	Therefore every edge of $C_n^k$ either lies within some $B_j$ or joins $B_j$ and $B_{j+1}$ for some $j$.

	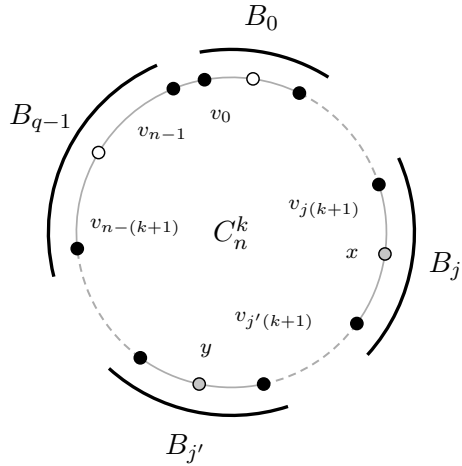
\begin{figure}[H]
		\centering
		\begin{tikzpicture}[
			dot/.style={
				circle,draw=black,fill=white,
				inner sep=0pt,minimum size=4.6pt,
				line width=0.55pt
			},
			enddot/.style={dot,fill=black},
			shadedot/.style={dot,fill=gray!45},
			every node/.style={font=\small}
			]
			\def\R{2.05}
			
			% Solid arcs within the displayed groups
			\foreach \a/\b in {
				100/64,18/-36,-78/-126,-174/-248
			}{
				\draw[gray!65,line width=0.65pt]
				(\a:\R)
				arc[start angle=\a,end angle=\b,
				radius=\R];
			}
			
			% Dashed arcs representing omitted groups
			\foreach \a/\b in {
				64/18,-36/-78,-126/-174
			}{
				\draw[gray!65,densely dashed,
				line width=0.75pt]
				(\a:\R)
				arc[start angle=\a,end angle=\b,
				radius=\R];
			}
			
			% Closing edge from v_{n-1} to v_0
			\draw[gray!65,line width=0.65pt]
			(-248:\R)
			arc[start angle=-248,end angle=-260,
			radius=\R];
			
			% B_0
			\node[enddot] at (100:\R) {};
			\node[dot]    at (82:\R) {};
			\node[enddot] at (64:\R) {};
			\draw[line width=1.25pt]
			(100:2.42)
			arc[start angle=100,end angle=58,
			radius=2.42];
			\node at (82:2.87) {$B_0$};
			
			% B_j
			\node[enddot]   at (18:\R) {};
			\node[shadedot] at (-8:\R) {};
			\node[enddot]   at (-36:\R) {};
			\draw[line width=1.25pt]
			(24:2.42)
			arc[start angle=24,end angle=-42,
			radius=2.42];
			\node at (-9:2.87) {$B_j$};
			
			% B_{j'}
			\node[enddot]   at (-78:\R) {};
			\node[shadedot] at (-102:\R) {};
			\node[enddot]   at (-126:\R) {};
			\draw[line width=1.25pt]
			(-72:2.42)
			arc[start angle=-72,end angle=-132,
			radius=2.42];
			\node at (-102:2.92) {$B_{j'}$};
			
			% B_{q-1}
			\node[enddot] at (-174:\R) {};
			\node[dot]    at (-211:\R) {};
			\node[enddot] at (-248:\R) {};
			\draw[line width=1.25pt]
			(-166:2.42)
			arc[start angle=-166,end angle=-246,
			radius=2.42];
			\node at (-211:2.94) {$B_{q-1}$};
			
			% Vertex labels
			\node[font=\scriptsize]
			at (95:1.55) {$v_0$};
			\node[font=\scriptsize]
			at (14:1.27) {$v_{j(k+1)}$};
			\node[font=\scriptsize]
			at (-10:1.62) {$x$};
			\node[font=\scriptsize]
			at (-64:1.30) {$v_{j'(k+1)}$};
			\node[font=\scriptsize]
			at (-102:1.60) {$y$};
			\node[font=\scriptsize]
			at (177:1.27) {$v_{n-(k+1)}$};
			\node[font=\scriptsize]
			at (125:1.59) {$v_{n-1}$};
			
			\node at (0,0) {$C_n^k$};
			
		\end{tikzpicture}
		\caption{A cyclic partition into groups of size $k+1$.}
		\label{fig:cyclic-partition}
	\end{figure}
	Now consider $B_j$ and $B_{j+1}$. For $0\le a,b\le k$, the path in $C_n$ from $v_{j(k+1)+a}$ to $v_{(j+1)(k+1)+b}$ following the cyclic order has length $k+1+b-a$. Thus this path has length
	at most $k$ if and only if $a>b$. For each $0\le\ell<\lceil\log_2(k+1)\rceil$
	and $p\ge0$, let
	\begin{equation}\label{eq:Qjlp}
		Q_{j,\ell,p}=\{v_{j(k+1)+a}:a\in I_{\ell,p}^1\}\cup\{v_{(j+1)(k+1)+b}:b\in I_{\ell,p}^0\}.
	\end{equation}
	By Lemma~\ref{lem:comparison}, each nonempty $Q_{j,\ell,p}$ is a clique in $C_n^k$.
	\begin{claim}\label{claim:disjoint}
		For each fixed $\ell$, the nonempty cliques $Q_{j,\ell,p}$ are pairwise vertex-disjoint.
	\end{claim}
	\begin{proof}
		Fix $\ell$ and let $v_{j(k+1)+a}\in B_j$. By \eqref{eq:Ilp1}, \eqref{eq:Ilp0} and \eqref{eq:Qjlp}, this vertex belongs to $Q_{j,\ell,p}$ if
		$\lfloor a/2^\ell\rfloor=2p+1$, and to $Q_{j-1,\ell,p}$ if $\lfloor a/2^\ell\rfloor=2p$. Since every nonnegative integer is uniquely of the form $2p$ or $2p+1$, each vertex belongs to exactly one of the sets $Q_{j,\ell,p}$. Consequently,
		\begin{equation*}
			V(C_n^k)=\bigcup_{j=0}^{q-1}\bigcup_{p\ge0}Q_{j,\ell,p},
		\end{equation*}
		where the nonempty sets $Q_{j,\ell,p}$ are pairwise vertex-disjoint.
	\end{proof}
	By Claim~\ref{claim:disjoint}, for each fixed $\ell$, the nonempty cliques $Q_{j,\ell,p}$ form the components of an equivalence subgraph of $C_n^k$.
	
	Now let $uv\in E(C_n^k)$ with $u$ and $v$ in distinct groups. By the preceding argument, a shortest $u$--$v$ path in $C_n$ crosses exactly one boundary, say from $B_j$ to $B_{j+1}$ in the cyclic order. Writing $u=v_{j(k+1)+a}$ and $v=v_{(j+1)(k+1)+b}$, we have $k+1+b-a\le k$, and hence $a>b$. By the construction in Lemma~\ref{lem:comparison}, $u,v\in Q_{j,\ell,p}$ for some $\ell,p$. Thus all edges joining distinct groups can be covered by $\lceil\log_2(k+1)\rceil$
	equivalence subgraphs. Together with the equivalence subgraph whose components are $B_0,\ldots,B_{q-1}$, we obtain
	\begin{equation*}
		\operatorname{eq}(C_n^k)\le 1+\lceil\log_2(k+1)\rceil.
	\end{equation*}
	The reverse inequality follows from Lemma~\ref{lem:lower}, completing the proof.
\end{proof}
\begin{remark}
	We illustrate the construction for $C_{12}^3$ in Figure~\ref{fig:C12-cover}. Label the vertices
	$v_1,\ldots,v_{12}$ in cyclic order, and let $B_0=\{v_1,v_2,v_3,v_4\}$, $B_1=\{v_5,v_6,v_7,v_8\}$ and $B_2=\{v_9,v_{10},v_{11},v_{12}\}$. The first equivalence subgraph has $B_0,B_1,B_2$ as its components, as shown in Figure~\ref{fig:C12-cover}(a).
	
	For $\ell=1$, we have $I_{1,0}^1=\{2,3\}$ and $I_{1,0}^0=\{0,1\}$, giving the cliques $	Q_{0,1,0}=\{v_3,v_4,v_5,v_6\}$, $Q_{1,1,0}=\{v_7,v_8,v_9,v_{10}\}$ and $Q_{2,1,0}=\{v_{11},v_{12},v_1,v_2\}$.
	
	For $\ell=0$, we have $I_{0,0}^1=\{1\}$, $I_{0,0}^0=\{0\}$, $I_{0,1}^1=\{3\}$ and $I_{0,1}^0=\{2\}$, giving $Q_{0,0,0}=\{v_2,v_5\}$, $Q_{0,0,1}=\{v_4,v_7\}$, $Q_{1,0,0}=\{v_6,v_9\}$, $Q_{1,0,1}=\{v_8,v_{11}\}$, $Q_{2,0,0}=\{v_{10},v_1\}$ and $Q_{2,0,1}=\{v_{12},v_3\}$.
	
	These are the components of the equivalence subgraphs in Figures~\ref{fig:C12-cover}(b) and~\ref{fig:C12-cover}(c), respectively.
	\begin{figure}[H]
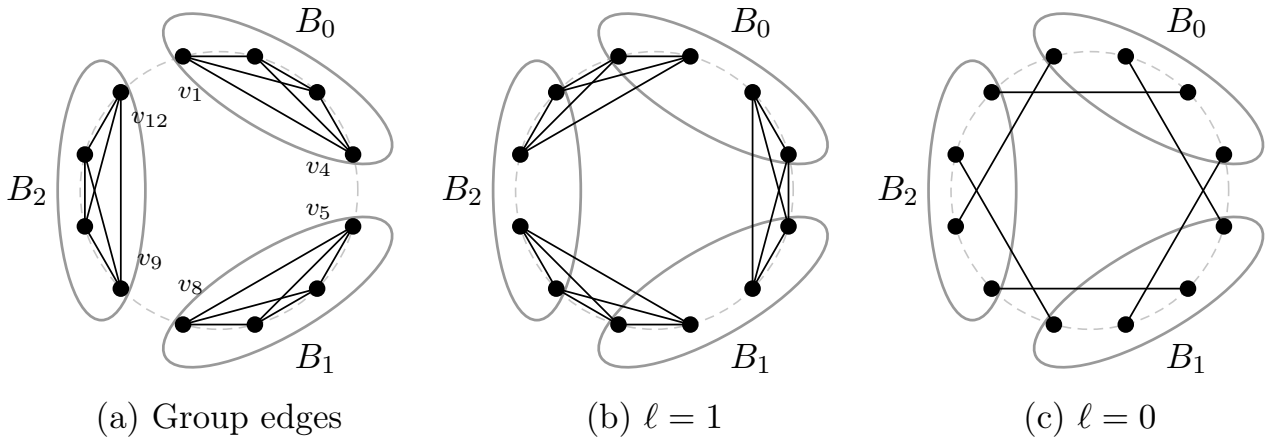

		\centering
		\begin{minipage}[t]{0.32\linewidth}\centering\resizebox{\linewidth}{!}{\cyclecover{0}}\end{minipage}\hfill
		\begin{minipage}[t]{0.32\linewidth}\centering\resizebox{\linewidth}{!}{\cyclecover{1}}\end{minipage}\hfill
		\begin{minipage}[t]{0.32\linewidth}\centering\resizebox{\linewidth}{!}{\cyclecover{2}}\end{minipage}
		\caption{The three equivalence subgraphs in the construction for $C_{12}^3$.}
		\label{fig:C12-cover}
	\end{figure}
\end{remark}
\begin{theorem}\label{thm:general-upper}
	For every $n\ge2k+2$,
	$\eqc(C_n^k)\le1+2\ceil{\log_2(k+1)}$.
\end{theorem}
\begin{proof}[\bf Proof]
	Let $P_n=v_0v_1\cdots v_{n-1}$. Since $P_n^k$ is an induced subgraph of $C_{q(k+1)}^k$ whenever
	$q(k+1)\ge n+k+1$, Lemma~\ref{lem:monotonicity} and Theorem~\ref{thm:main-divisible} give
	\begin{equation}\label{eq:eqpnk}
		\operatorname{eq}(P_n^k)\le 1+\lceil\log_2(k+1)\rceil.
	\end{equation}
	It remains to cover the edges of $C_n^k$ not belonging to $P_n^k$. Put $A_a=v_{n-k-1+a}$ and $B_b=v_b$, where $0\le a,b\le k$. Since $n\ge 2k+2$, the vertices $A_a$ and $B_b$ form two vertex-disjoint cliques in $C_n^k$. Moreover, the edges of $C_n^k$ not belonging to $P_n^k$ are exactly $A_aB_b$ with $a>b$, since the path between them through $v_{n-1}v_0$ has length $k+1-a+b$.
	
	By Lemma~\ref{lem:comparison}, these edges can be covered by $\lceil\log_2(k+1)\rceil$ equivalence subgraphs. Using \eqref{eq:eqpnk}, we have
	\begin{equation*}
		\operatorname{eq}(C_n^k)\le \operatorname{eq}(P_n^k)+\lceil\log_2(k+1)\rceil\le 1+2\lceil\log_2(k+1)\rceil,
	\end{equation*}
	as required.
\end{proof}
Combining Lemma~\ref{lem:lower} and Theorem~\ref{thm:general-upper}, this proves Theorem~\ref{thm:main}.
\section{Proof of Theorem~\ref{thm:exact-values}}\label{sec:1.3}
In this section, we prove Theorem~\ref{thm:exact-values}. We first use induced subgraph embeddings to obtain an upper bound and determine the exact value of $\eqc(C_n^k)$ in the first case. We then establish a stronger lower bound when $n$ is odd and combine it with the upper bound to prove the second case.
\begin{lemma}\label{lem:embedding}
	Let $n\ge 2k+2$, $N\ge n$, and $K\ge1$ be integers. If $kN\le Kn$ and $(k+1)N\ge(K+1)n$, then $C_n^k$ is an induced subgraph of $C_N^K$.
\end{lemma}
\begin{proof}[\bf Proof]
	For $0\le i<n$, replace the edge $v_iv_{i+1}$ of $C_n$ by a path of length $\left\lfloor{(i+1)N}/{n}\right\rfloor-\left\lfloor{iN}/{n}\right\rfloor$. Since $N\ge n$, these lengths are positive and satisfy
	\begin{equation*}
		\sum_{i=0}^{n-1} \left( \left\lfloor \frac{(i+1)N}{n} \right\rfloor - \left\lfloor \frac{iN}{n} \right\rfloor \right) = \left\lfloor \frac{nN}{n} \right\rfloor = N.
	\end{equation*}
	Thus the resulting graph is $C_N$. Moreover, for $1\le d<n$, the path $v_iv_{i+1}\cdots v_{i+d}$ in $C_n$ is replaced by a path in $C_N$ whose length satisfies
	\begin{equation*}
		\left\lfloor \frac{dN}{n} \right\rfloor \le \left\lfloor \frac{(i+d)N}{n} \right\rfloor - \left\lfloor \frac{iN}{n} \right\rfloor \le \left\lceil \frac{dN}{n} \right\rceil.
	\end{equation*}
	For any distinct vertices $u,v\in V(C_n^k)$, let $d_{C_n}(u, v)$ denote the length of a shortest $u$--$v$ path in $C_n$. If $uv\in E(C_n^k)$, then $d_{C_n}(u, v)\leq k$. Thus
	\begin{equation*}
		d_{C_N}(u, v) \leq \left\lceil \frac{N d_{C_n}(u, v)}{n} \right\rceil \leq \left\lceil \frac{kN}{n} \right\rceil \leq K.
	\end{equation*}
	Hence $uv\in E(C_N^K)$. Conversely, if $u$ and $v$ are distinct and $uv\notin E(C_n^k)$, then both $u$--$v$ paths in $C_n$ have length at least $k+1$. Thus
	\begin{equation*}
		d_{C_N}(u, v) \geq \left\lfloor \frac{N d_{C_n}(u, v)}{n} \right\rfloor\geq \left\lfloor \frac{(k + 1)N}{n} \right\rfloor\geq K + 1.
	\end{equation*}
	Thus $uv\notin E(C_N^K)$. Therefore $C_n^k$ is an induced subgraph of $C_N^K$, as required.
\end{proof}
We now apply Lemma~\ref{lem:embedding} to obtain a sharper upper bound for sufficiently large $n$.
\begin{proposition}\label{prop:quadratic}
	If $n\ge\max\{2k+2,2k^2\}$, then
	$\eqc(C_n^k)\le2+\ceil{\log_2(k+1)}$.
\end{proposition}
\begin{proof}[\bf Proof]
	Let $m=\lceil n/(k+1)\rceil$. Clearly, $m(k+1)\ge n$. Since
	$n\ge2k^2$ and $m\le(n+k)/(k+1)$, we have
	\begin{equation*}
			m(2k+1)\le\frac{(n+k)(2k+1)}{k+1}=2n+\frac{2k^2+k-n}{k+1}\le2n+\frac{k}{k+1}<2n+1.
	\end{equation*}
Since $m(2k+1)$ is an integer, it follows that $m(2k+1)\le2n$. By Lemma~\ref{lem:embedding}, $C_n^k$ is an induced subgraph of $C_{m(2k+1)}^{2k}$. Hence, by Lemma~\ref{lem:monotonicity} and Theorem~\ref{thm:main-divisible}, we have
	\begin{equation*}
		\operatorname{eq}(C_n^k)\le \operatorname{eq}(C_{m(2k+1)}^{2k})=1+\lceil\log_2(2k+1)\rceil\le2+\lceil\log_2(k+1)\rceil.
	\end{equation*}
	This completes the proof.
\end{proof}
Using a similar approach, we next obtain the exact value of $\eqc(C_n^k)$ under a stronger condition on $n$.
\begin{theorem}\label{thm:eq=1}
	Let $k\ge1$ be an integer satisfying $k+1<2^{\lceil\log_2(k+1)\rceil}$, and let $c_k^*=\max\{2k+2,k^2(k+2)\}$. Then, for every integer $n\ge c_k^*$,
	\begin{equation*}
	\eqc(C_n^k)=1+\lceil\log_2(k+1)\rceil.
	\end{equation*}
\end{theorem}
\begin{proof}[\bf Proof]
	Let $m=\lceil n/(k+1)\rceil$ and $r=2^{\lceil\log_2(k+1)\rceil}$. Since $k+1$ is not a power of two, we have $r-k-1\ge1$. Thus
	\begin{equation*}
		\frac{r}{r-k-1}=1+\frac{k+1}{r-k-1}\le k+2.
	\end{equation*}
	By $n\ge k^2(k+2)$ and $m\le(n+k)/(k+1)$, we obtain
	\begin{equation*}
		(r-1)n-kmr\ge\frac{(r-k-1)n-k^2r}{k+1}\ge0.
	\end{equation*}
	Moreover, $mr\ge n$ and $(k+1)mr\ge rn$. By Lemma~\ref{lem:embedding}, $C_n^k$ is an induced subgraph of $C_{mr}^{r-1}$. Hence, by Lemma~\ref{lem:monotonicity} and Theorem~\ref{thm:main-divisible}, we have
	\begin{equation*}
		\eqc(C_n^k)\le\eqc(C_{mr}^{r-1})=1+\lceil\log_2r\rceil=1+\lceil\log_2(k+1)\rceil.
	\end{equation*}
	The reverse inequality follows from Lemma~\ref{lem:lower}, completing the proof.
\end{proof}
We next consider the case $k+1=2^{\lceil\log_2(k+1)\rceil}$ and establish a stronger lower bound when $n$ is odd.
\begin{theorem}\label{thm:odd-power}
	Let $k\ge1$ be an integer satisfying $k+1=2^{\lceil\log_2(k+1)\rceil}$. For every odd integer $n\ge2k+2$, we have
	\begin{equation*}
		\eqc(C_n^k)\ge2+\lceil\log_2(k+1)\rceil.
	\end{equation*}
\end{theorem}
\begin{proof}[\bf Proof]
	Suppose that $C_n^k$ can be covered by $t=\lceil\log_2(k+1)\rceil+1$ equivalence subgraphs $G_1,\ldots,G_t$. Let $\mathbb F$ be an infinite field of characteristic two. For each $j\in[t]$,
	assign distinct elements of $\mathbb F$ to the components of $G_j$, and let $c_j(v)$
	denote the element assigned to the component containing $v$. Since each component of $G_j$ is a clique, for any distinct vertices $u,v$, we have $c_j(u)=c_j(v)\iff uv\in E(G_j)$. As $G_1,\ldots,G_t$ cover $C_n^k$, it follows that
	\begin{equation*}
		\Psi(u,v)=\prod_{j=1}^t(c_j(u)+c_j(v))
		\begin{cases}
			=0,&u=v\text{ or }uv\in E(C_n^k),\\
			\ne0,&u\neq v\text{ and }uv\notin E(C_n^k).
		\end{cases}
	\end{equation*}
	For each vertex $v_i$, define
	\begin{equation}\label{eq:xi}
		x_i=\left(\prod_{j\in S}c_j(v_i)\right)_{S\subseteq[t]}\in\mathbb F^{2(k+1)},
	\end{equation}
where the coordinates are indexed by the subsets $S\subseteq[t]$ and the empty product is $1$. For an arbitrary vector $x=(x_S)_{S\subseteq[t]}$, define
	\begin{equation*}
		Q(x)=\sum_{\substack{S\subseteq[t]\\1\in S}}x_Sx_{[t]\setminus S},
	\end{equation*}
	and let $B(x,y)=Q(x+y)-Q(x)-Q(y)$ be its associated bilinear form. Expanding the definition gives
\begin{align*}
	B(x,y)
	&=\sum_{\substack{S\subseteq[t]\\1\in S}}
	\Bigl(
	(x_S+y_S)(x_{[t]\setminus S}+y_{[t]\setminus S})
	-x_Sx_{[t]\setminus S}
	-y_Sy_{[t]\setminus S}
	\Bigr)\\
	&=\sum_{\substack{S\subseteq[t]\\1\in S}}
	\left(
	x_Sy_{[t]\setminus S}
	+y_Sx_{[t]\setminus S}
	\right)\\
	&=\sum_{\substack{S\subseteq[t]\\1\in S}}
	x_Sy_{[t]\setminus S}
	+\sum_{\substack{S\subseteq[t]\\1\notin S}}
	x_Sy_{[t]\setminus S}\\
	&=\sum_{S\subseteq[t]}x_Sy_{[t]\setminus S}.
\end{align*}
Thus the matrix of $B$ is a permutation matrix, and hence $B$ is nondegenerate. By \eqref{eq:xi}, we have
\begin{equation*}
B(x_i,x_j)=\sum_{S\subseteq[t]}\left(\prod_{a\in S}c_a(v_i)\right)\left(\prod_{a\notin S}c_a(v_j)\right)=\prod_{a=1}^t\bigl(c_a(v_i)+c_a(v_j)\bigr)=\Psi(v_i,v_j).
\end{equation*}
Consequently,
\begin{equation*}
B(x_i,x_j)
\begin{cases}
	=0, & i=j\text{ or }v_iv_j\in E(C_n^k),\\
	\ne0, & i\ne j\text{ and }v_iv_j\notin E(C_n^k).
\end{cases}
\end{equation*}
Moreover, since $2^{t-1}=k+1$ is even and $\mathbb F$ has characteristic two, for every $i$ we have
\begin{align}\label{eq:Qxi}
Q(x_i)&=\sum_{\substack{S\subseteq[t]\\1\in S}}(x_i)_S(x_i)_{[t]\setminus S}=\sum_{\substack{S\subseteq[t]\\1\in S}}
\left(\prod_{j\in S}c_j(v_i)\right)
\left(\prod_{j\in[t]\setminus S}c_j(v_i)\right)=\sum_{\substack{S\subseteq[t]\\1\in S}}
\prod_{j=1}^t c_j(v_i)\notag\\
&=2^{t-1}\prod_{j=1}^t c_j(v_i)=(k+1)\prod_{j=1}^t c_j(v_i)=0.
\end{align}
For each $i$, let
\begin{equation*}
W_i=\operatorname{span}\{x_i,x_{i+1},\ldots,x_{i+k}\},
\end{equation*}
where vertex indices are taken modulo $n$. Since $v_i,\ldots,v_{i+k}$ form a clique in $C_n^k$, we have $B(x_{i+a},x_{i+b})=0$ for $0\le a,b\le k$. Moreover, by \eqref{eq:Qxi}, we have $Q(x_{i+a})=0$ for every $0\le a\le k$. Thus, for any $x=\sum_{a=0}^k\lambda_a x_{i+a}\in W_i$, we obtain
\begin{equation*}
Q(x)=\sum_{a=0}^k\lambda_a^2Q(x_{i+a})+\sum_{0\le a<b\le k}\lambda_a\lambda_b B(x_{i+a},x_{i+b})=0.
\end{equation*}
Hence $Q$ vanishes on $W_i$.
\begin{claim}\label{claim:dimWi=k+1}
$\dim W_i=k+1$.
\end{claim}
\begin{proof}
Consider the $(k+1)\times(k+1)$ matrix $M=\bigl(B(x_{i+a},x_{i+k+1+b})\bigr)_{0\le a,b\le k}$.  If $a>b$, then $1\le k+1+b-a\le k$, so $v_{i+a}v_{i+k+1+b}\in E(C_n^k)$ and hence $M_{ab}=0$. If $a=b$, then $d_{C_n}(v_{i+a},v_{i+k+1+a})=k+1$, as $n\ge2k+2$. Thus $M_{aa}\ne0$. Consequently, $M$ is upper triangular with nonzero diagonal, and hence $k+1=\operatorname{rank}M\le\dim W_i\le k+1$.
\end{proof}
\begin{claim}\label{claim:dim(WicapWi+1)=k}
$\dim(W_i\cap W_{i+1})=k$.
\end{claim}
\begin{proof}
By Claim~\ref{claim:dimWi=k+1}, $W_i$ and $W_{i+1}$ contain the $k$ linearly independent vectors $x_{i+1},\ldots,x_{i+k}$. Hence $\dim(W_i\cap W_{i+1})\ge k$. We claim that $W_i\ne W_{i+1}$. Otherwise, $x_{i+k+1}\in W_i$, so
\begin{equation*}
x_{i+k+1}=\sum_{a=0}^k\lambda_a x_{i+a}
\end{equation*}
for some $\lambda_0,\ldots,\lambda_k\in\mathbb F$. Since $B(x_i,x_{i+a})=0$ for $0\le a\le k$, we obtain
\begin{equation*}
B(x_i,x_{i+k+1})=\sum_{a=0}^k\lambda_a B(x_i,x_{i+a})=0.
\end{equation*}
However, $d_{C_n}(v_i,v_{i+k+1})=k+1$, so $B(x_i,x_{i+k+1})\ne0$, a contradiction. Since both spaces have dimension $k+1$, we conclude that $\dim(W_i\cap W_{i+1})=k$.
\end{proof}
\begin{claim}\label{claim:dimWU-dimWv=1}
Let $U,V,W$ be $(k+1)$-dimensional subspaces of $\mathbb F^{2(k+1)}$ on which $Q$ vanishes.
If $\dim(U\cap V)=k$, then 
\begin{equation*}
|\dim(W\cap U)-\dim(W\cap V)|=1.
\end{equation*}
\end{claim}
\begin{proof}
Let $L=U\cap V$ and $r=\dim(W\cap L)$. Since $Q$ vanishes on $U,V,W$, so does $B$. As $B$ is nondegenerate and these subspaces have dimension $k+1$, we have $U^\perp=U$, $V^\perp=V$ and $W^\perp=W$. Moreover, since $\dim L=k$, we have
\begin{equation*}
\dim(U+V)=k+2=\dim L^\perp.
\end{equation*}
As $U+V\subseteq L^\perp$, it follows that $U+V=L^\perp$. Since $W^\perp=W$ and $(W+L)^\perp=W^\perp\cap L^\perp$, we have
\begin{align*}
\dim(W\cap(U+V))
&=\dim(W\cap L^\perp)
\\&=\dim((W+L)^\perp)
\\&=2(k+1)-\dim(W+L)
\\&=2(k+1)-(2k+1-r)
\\&=r+1.
\end{align*}
Choose $w\in W\cap(U+V)$ with $w\notin L$, and write $w=u+v$, where $u\in U$ and $v\in V$.
Since $Q$ vanishes on $U,V,W$, we have
\begin{equation*}
B(u,v)=Q(w)-Q(u)-Q(v)=0.
\end{equation*}
We next show that $u\in L$ or $v\in L$. Suppose that $u\notin L$ and $v\notin L$. Since $L\subseteq U$ and $B$ vanishes on $U$, we have $B(u,z)=0$ for every $z\in L$. Moreover, as $\dim L=k$, $\dim V=k+1$, and $v\in V\setminus L$, we have $V=L+\operatorname{span}\{v\}$. Together with $B(u,v)=0$, this gives
$u\in V^\perp=V$. Since $u\in U$, we obtain $u\in U\cap V=L$, a contradiction. Thus $u\in L$ or $v\in L$. As $w\notin L$, exactly one of these holds, and hence $w$ belongs to exactly one of $U$ and $V$.

Without loss of generality, suppose that $w\in U\setminus L$. Since $w\in W$ and $\dim(W\cap L)=r$, we have $\dim(W\cap U)\ge r+1$. On the other hand, since $W\cap U\subseteq W\cap(U+V)$ and
$\dim(W\cap(U+V))=r+1$, we obtain $\dim(W\cap U)\le r+1$. Thus $\dim(W\cap U)=r+1$.

We show that $\dim(W\cap V)=r$. Suppose that $\dim(W\cap V)=r+1$. Since $\dim(W\cap L)=r$, there exists $v'\in(W\cap V)\setminus L$. As $w\in U$ and $L\subseteq U$, we have $B(w,z)=0$ for every $z\in L$. Moreover, since $\dim V=k+1$, $\dim L=k$, and $v'\in V\setminus L$, we have $V=L+\operatorname{span}\{v'\}$. If $B(w,v')=0$, then $w\in V^\perp=V$, contrary to $w\in U\setminus L$. Thus $B(w,v')\ne0$. However, since $w,v'\in W$ and $B$ vanishes on $W$, we have $B(w,v')=0$, a contradiction. Hence $|\dim(W\cap U)-\dim(W\cap V)|=1$, as required.
\end{proof}
By Claim~\ref{claim:dim(WicapWi+1)=k}, we have $\dim(W_i\cap W_{i+1})=k$. Applying Claim~\ref{claim:dimWU-dimWv=1} with $U=W_i$, $V=W_{i+1}$ and $W=W_0$, we obtain
\begin{equation*}
\left|\dim(W_0\cap W_{i+1})-\dim(W_0\cap W_i)\right|=1.
\end{equation*}
Hence
\begin{equation*}
\dim(W_0\cap W_{i+1})\equiv\dim(W_0\cap W_i)+1\pmod2.
\end{equation*}
Applying this relation successively for $i=0,\ldots,n-1$, we obtain
\begin{equation*}
\dim(W_0\cap W_n)\equiv\dim(W_0\cap W_0)+n\pmod2.
\end{equation*}
Since vertex indices are taken modulo $n$, we have $W_n=W_0$. Therefore, $k+1\equiv k+1+n\pmod2$, which implies that $n$ is even, contrary to our assumption. Thus $C_n^k$ cannot be covered by $t$ equivalence subgraphs, and hence $\eqc(C_n^k)\ge t+1=\lceil\log_2(k+1)\rceil+2$.
\end{proof}
\begin{theorem}\label{thm:eq=2}
Let $k\ge1$ be an integer satisfying $k+1=2^{\lceil\log_2(k+1)\rceil}$, and let $c_k^*=\max\{2k+2,k^2(k+2)\}$. For every odd integer $n\ge c_k^*$, we have
	\begin{equation*}
		\eqc(C_n^k)=\ceil{\log_2(k+1)}+2.
	\end{equation*}
\end{theorem}
\begin{proof}[\bf Proof]
Combining Proposition~\ref{prop:quadratic} and Theorem~\ref{thm:odd-power} gives the result.
\end{proof}
Theorem~\ref{thm:exact-values} follows from Theorems~\ref{thm:main-divisible}, \ref{thm:eq=1} and \ref{thm:eq=2}.
\section{Proof of Theorem~\ref{thm:main-gap}}\label{sec:1.4}
In this section, we prove Theorem~\ref{thm:main-gap}. We first show how a covering of a smaller power of a cycle can be used to cover a larger one. We then combine this construction with Lemma~\ref{lem:embedding} to obtain a sharper upper bound when $n$ is sufficiently large.
\begin{lemma}\label{lem:pair-lift}
If $p\ge1$ and $m\ge2p+2$, then $\eqc(C_{2m}^{2p+1})\le1+\eqc(C_m^p)$.
\end{lemma}
\begin{proof}[\bf Proof]
Let $t=\operatorname{eq}(C_m^p)$, and take a covering of $C_m^p$ by $t$ equivalence subgraphs. For each vertex $i$, consider the pair $\{2i,2i+1\}$ in $C_{2m}^{2p+1}$. In each equivalence subgraph, replace every component by the subgraph induced by the union of the corresponding pairs.

For any $i,j$ with $d_{C_m}(i,j)\le p$ and $a,b\in\{0,1\}$, we have
\begin{equation*}
d_{C_{2m}}(2i+a,2j+b)\le2d_{C_m}(i,j)+1\le2p+1.
\end{equation*}
Thus the resulting sets are cliques in $C_{2m}^{2p+1}$, and their vertex-disjointness is preserved. Hence we obtain $t$ equivalence subgraphs of $C_{2m}^{2p+1}$.

These subgraphs together cover precisely the edges within each pair and all edges joining pairs corresponding to adjacent vertices of $C_m^p$. Since $C_m^p$ is $2p$-regular, the graph formed by these edges is $(4p+1)$-regular, as $1+2(2p)=4p+1$. On the other hand, $C_{2m}^{2p+1}$ is $(4p+2)$-regular. Consequently, the remaining edges form a perfect matching, which can be covered by one additional equivalence subgraph. Therefore,
\begin{equation*}
\operatorname{eq}(C_{2m}^{2p+1})\le1+\operatorname{eq}(C_m^p),
\end{equation*}
as required.
\end{proof}
We now combine Lemma~\ref{lem:pair-lift} with the induced subgraph embedding in Lemma~\ref{lem:embedding} to obtain a more general result.
\begin{lemma}\label{lem:reduction}
Let $n\ge 2k+2$. Then there exist integers $p\ge1$ and $m\ge2p+2$ such that $C_n^k$ is an induced subgraph of $C_{2m}^{2p+1}$ and ${m}/{(p+1)}\ge{n}/{(k+1)}$. In particular,
\begin{equation*}
\operatorname{eq}(C_n^k)\le1+\operatorname{eq}(C_m^p).
\end{equation*}
\end{lemma}
\begin{proof}[\bf Proof]
Choose integers $p$ and $m$ such that
\begin{equation}\label{eq:p+1,m}
p+1=\left\lceil\frac{k+1}{2}+\frac{k(k+1)}{n}\right\rceil,
\qquad
m=\left\lceil\frac{n(p+1)}{k+1}\right\rceil.
\end{equation}
Clearly, $p\ge1$. Since $n\ge2k+2$, we have
\begin{equation*}
\frac{m}{p+1}\ge\frac{n}{k+1}\ge2,
\qquad
2m\ge\frac{2n(p+1)}{k+1}\ge n.
\end{equation*}
Thus $m\ge2p+2$. Moreover, by \eqref{eq:p+1,m} we have
\begin{equation*}
	2km<\frac{2nk(p+1)}{k+1}+2k\le n(2p+1),
\end{equation*}
and $2m(k+1)\ge2n(p+1)$. By Lemma~\ref{lem:embedding}, $C_n^k$ is an induced subgraph of $C_{2m}^{2p+1}$. Hence, by Lemmas~\ref{lem:monotonicity} and~\ref{lem:pair-lift}, we have
\begin{equation*}
\operatorname{eq}(C_n^k)\le\operatorname{eq}(C_{2m}^{2p+1})\le1+\operatorname{eq}(C_m^p),
\end{equation*}
as required.
\end{proof}
We now apply Lemma~\ref{lem:reduction} repeatedly to obtain a sharper upper bound for sufficiently large $n$.
\begin{proposition}\label{prop:logarithmic-order}
If $n\ge(k+1)\max\{41,5\ceil{\log_2(k+1)}\}$, then $\eqc(C_n^k)\le2+\ceil{\log_2(k+1)}$.
\end{proposition}
\begin{proof}[\bf Proof]
Let $s=\lceil\log_2(k+1)\rceil$. If $s=1$, then $k=1$. Every even cycle is the union of two matchings, while every odd cycle can be covered by three matchings. Thus the result holds for
$s=1$. Assume $s\ge2$.

Starting with $(n_0,k_0)=(n,k)$, apply Lemma~\ref{lem:reduction} $s-2$ times. For $0\le i<s-2$, choosing 
\begin{equation}\label{eq:k+1,n+1}
k_{i+1}+1=\left\lceil\frac{k_i+1}{2}+\frac{k_i(k_i+1)}{n_i}\right\rceil,\qquad n_{i+1}=\left\lceil\frac{n_i(k_{i+1}+1)}{k_i+1}\right\rceil.
\end{equation}
By Lemma~\ref{lem:reduction}, the ratios satisfy
\begin{equation}\label{eq:n>max 42,5s}
\frac{n_i}{k_i+1}\ge\frac{n}{k+1}\ge\max\{41,5s\}
\end{equation}
at every step. Hence the lemma remains applicable, and
\begin{equation}\label{eq:iterative-bound}
	\operatorname{eq}(C_n^k)\le(s-2)+\operatorname{eq}(C_{n_{s-2}}^{k_{s-2}}).
\end{equation}
Moreover, by \eqref{eq:k+1,n+1} and \eqref{eq:n>max 42,5s}, we have
\begin{equation*}
	k_{i+1}+1<\frac{k_i+1}{2}+\frac{k_i(k_i+1)}{n_i}+1<\left(\frac12+\frac1{5s}\right)(k_i+1)+1.
\end{equation*}
Since $k+1\le2^s$, for $s\ge3$, repeated application gives
\begin{align}\label{eq:ks-2+1<}
k_{s-2}+1
&<\left(\frac12+\frac1{5s}\right)^{s-2}(k+1)+\sum_{i=0}^{s-3}\left(\frac12+\frac1{5s}\right)^i\notag
\\&<2^s\left(\frac12+\frac1{5s}\right)^{s-2}+\frac{1}{\frac12-\frac1{5s}}\notag
\\&=4\left(1+\frac2{5s}\right)^{s-2}+\frac{10s}{5s-2}.
\end{align}
For $s=2$, this bound also follows directly from $k+1\le4$. Since $\left(1+\frac2{5s}\right)^s
<e^{2/5}<\frac32$, we have
\begin{equation}\label{eq:4s-2<30s}
4\left(1+\frac2{5s}\right)^{s-2}<\frac6{(1+2/(5s))^2}\le\frac{30s}{5s+4}.
\end{equation}
Combining \eqref{eq:ks-2+1<} and \eqref{eq:4s-2<30s}, we obtain
\begin{equation*}
k_{s-2}+1<\frac{30s}{5s+4}+\frac{10s}{5s-2}=8-\frac{24}{5s+4}+\frac4{5s-2}<8,
\end{equation*}
where the last inequality follows from $s\ge2$. Since $k_{s-2}$ is an integer, we have $1\le k_{s-2}\le6$. 

It remains to show that $\operatorname{eq}(C_{n_{s-2}}^{k_{s-2}})\le4$ whenever $1\le k_{s-2}\le6$ and
$n_{s-2}\ge41(k_{s-2}+1)$. Let $j=\left\lceil {n_{s-2}}/{(k_{s-2}+1)}\right\rceil$. If $j=41$, then $n_{s-2}=41(k_{s-2}+1)$, and hence $7n_{s-2}-8jk_{s-2}=41(7-k_{s-2})>0$. If $j\ge42$, then $n_{s-2}\ge(j-1)(k_{s-2}+1)+1$, so $7n_{s-2}-8jk_{s-2}\ge j(7-k_{s-2})-7k_{s-2}\ge49(6-k_{s-2})\ge0$.  Moreover, $n_{s-2}\le j(k_{s-2}+1)<8j$ and $8j(k_{s-2}+1)\ge8n_{s-2}$. By Lemma~\ref{lem:embedding}, $C_{n_{s-2}}^{k_{s-2}}$ is an induced
subgraph of $C_{8j}^7$. Hence, by Lemma~\ref{lem:monotonicity} and Theorem~\ref{thm:main-divisible}, we have
\begin{equation}\label{eq:eqC<=4}
\operatorname{eq}(C_{n_{s-2}}^{k_{s-2}})
\le\operatorname{eq}(C_{8j}^7)=4.
\end{equation}
Substituting \eqref{eq:eqC<=4} into \eqref{eq:iterative-bound}, we obtain
\begin{equation*}
	\operatorname{eq}(C_n^k)\le(s-2)+4=2+\lceil\log_2(k+1)\rceil,
\end{equation*}
as required.
\end{proof}
Combining Lemma~\ref{lem:lower} and Proposition~\ref{prop:logarithmic-order}, this proves Theorem~\ref{thm:main-gap}.

\section*{Declaration on the use of AI}
The author used ChatGPT to assist in discussing proof strategies, checking proofs, and improving exposition.

\end{document}